\documentclass[11pt, a4paper]{article}
\usepackage{amssymb, amsmath, amsthm, amsfonts,enumerate}
\usepackage{amsmath,setspace,scalefnt}
\usepackage{fancyhdr}
\usepackage[british]{babel}

\usepackage[usenames,dvipsnames,table]{xcolor}
\usepackage{tikz}
\usetikzlibrary{decorations.pathreplacing}
\usetikzlibrary{decorations.pathmorphing}
\usetikzlibrary{snakes}
\usetikzlibrary{calc,patterns,angles,quotes}
\usetikzlibrary{scopes}
\usetikzlibrary{backgrounds}
\usepackage{bbm}

\usepackage{kotex}
\usepackage{todonotes}
\usepackage{bbm,dsfont}

\usepackage{hyperref}
\usepackage[margin=2.8cm]{geometry}
\usepackage{enumitem,verbatim}
\usepackage[capitalise]{cleveref}

\newtheorem{theorem}{Theorem}[section]

\newtheorem{lemma}[theorem]{Lemma}

\theoremstyle{definition}

\newtheorem{claim}[theorem]{Claim}

\newenvironment{poc}{\begin{proof}[Proof of claim]}{\end{proof}}

\newcommand{\bP}{\mathbb{P}}
\newcommand{\bE}{\mathbb{E}}

\newcommand{\rV}{\mathrm{Var}}

\newcommand{\vol}{\mathrm{vol}}
\newcommand{\Sd}{\mathbb{S}_{d-1}}

\newcommand{\eps}{\varepsilon}
\newcommand{\bN}{\mathbb{N}}

\newcommand{\bR}{\mathbb{R}}

\newcommand{\dd}{\,\mathrm{d}}

\definecolor{darkblue}{rgb}{0,0,0.5}
\definecolor{armygreen}{rgb}{0.29, 0.33, 0.13}
\definecolor{darkmagenta}{rgb}{0.55, 0.0, 0.55}
\definecolor{lightseagreen}{rgb}{0.13, 0.7, 0.67}
\definecolor{darktangerine}{rgb}{1.0, 0.66, 0.07}
\definecolor{deepcarmine}{rgb}{0.66, 0.13, 0.24}
\definecolor{darkblue}{rgb}{0.0, 0.0, 0.55}
\definecolor{darkraspberry}{rgb}{0.53, 0.15, 0.34}
\definecolor{darkpastelgreen}{rgb}{0.01, 0.75, 0.24}
\definecolor{darkcyan}{rgb}{0.0, 0.55, 0.55}

\newcommand{\hide}[1]{}
\newcommand{\me}{\mathrm{e}}

\title{New lower bounds on kissing numbers and spherical codes in high dimensions}

\author{
	Irene Gil Fern\'andez
	\thanks{Mathematics Institute and DIMAP, University of Warwick, UK. Email: {\tt irene.gil-fernandez@warwick.ac.uk}. IGF is supported by the Warwick Mathematics Institute Centre for Doctoral Training, and gratefully acknowledges funding from the University of Warwick and the UK Engineering and Physical Sciences Research Council (Grant number: EP/TS1794X/1).}
	\and
	Jaehoon Kim
	\thanks{Department of Mathematical Sciences, KAIST, South Korea. Email: {\tt jaehoon.kim@kaist.ac.kr}. J.K. was supported by the POSCO Science Fellowship of POSCO TJ Park Foundation, and by the National Research Foundation of Korea (NRF) grant funded by  the Korea government(MSIT) No. RS-2023-00210430.}
	\and
	Hong Liu
	\thanks{Extremal Combinatorics and Probability Group (ECOPRO), Institute for Basic Science (IBS), Daejeon, South Korea. Email: {\tt hongliu@ibs.re.kr}. H.L. was supported by the Institute for Basic Science (IBS-R029-C4) and the UK Research and Innovation Future Leaders Fellowship MR/S016325/1.}
	\and
	Oleg Pikhurko
	\thanks{Mathematics Institute and DIMAP, University of Warwick, UK.  Email: {\tt O.Pikhurko@warwick.ac.uk}. O.P. was supported by ERC Advanced Grant 101020255 and Leverhulme Research Project Grant RPG-2018-424.}
}

\date{\today}

\begin{document}
	
	\maketitle
	
	\begin{abstract}
	Let the \emph{kissing number} $K(d)$ be the maximum number of non-overlapping unit balls in $\mathbb R^d$ that can
	touch a given unit ball. Determining or estimating the number $K(d)$ has a long history, with the value of $K(3)$
	being the subject of a famous discussion between Gregory and Newton in 1694. We prove that, as the dimension $d$ goes to infinity, 
	$$
	K(d)\ge (1+o(1)){\frac{\sqrt{3\pi}}{4\sqrt2}}\,\log\frac{3}{2}\cdot d^{3/2}\cdot \Big(\frac{2}{\sqrt{3}}\Big)^{d},
	$$
 thus improving the previously best known bound of Jenssen, Joos and Perkins~\cite{JenssenJoosPerkins18} by a factor of
 $\log(3/2)/\log(9/8)+o(1)=3.442...$\,. Our proof is based on the novel approach from~\cite{JenssenJoosPerkins18} that uses the hard sphere model of an appropriate fugacity. Similar constant-factor improvements in lower bounds
 are also obtained for general spherical codes, as well as for the expected density of random sphere packings in the Euclidean space $\mathbb R^d$.
	\end{abstract}

\section{Introduction}\label{sec-intro}

The \textit{kissing number} in dimension $d$, denoted by $K(d)$, is the maximum number of non-overlapping  (i.e., having disjoint interiors) unit balls in the Euclidean $d$-dimensional space $\mathbb R^d$ that can touch a given unit ball.

It is easy to see that $K(1)=2$ and $K(2)=6$. Whether $K(3)$ is 12 or larger was the
subject of a famous discussion between David Gregory and Isaac Newton in 1694;
see e.g.~\cite{PfenderZiegler2004} for a historic account.
This problem was finally resolved  in 1953, by Sch\"{u}tte and van der Waerden~\cite{Waerdenvander1952/53} who proved that $K(3)=12$. New proofs of $K(3)=12$ (see e.g. Maehara~\cite{Maehara2001}, B\"{o}r\"{o}czky~\cite{Borozsky2003} and Anstreicher~\cite{Anstreicher2004})
were discovered more recently, as it continues to be a problem of interest.  The only other known
values are $K(4)=24$, proved by Musin~\cite{Musin2008} using a modification of Delsarte's method; and $K(8)=240$ and $K(24)=196560$ were proved in 1979 by Levenshtein~\cite{Lev79} and, independently, by Odlyzko and Sloane~\cite{Odlyzko1979} using Delsarte's method. For a survey of kissing numbers, see~\cite{boyvalenkov2015survey}.

In large dimensions, the best known bounds are exponentially far apart. 
For the upper bound, the first exponential improvement over the easy bound of $O(2^d)$ (coming from a volume argument) was obtained by Rankin \cite{Rankin55}:
$$K(d)\leq (1+o(1))\frac{\sqrt{\pi}}{2\sqrt{2}}\cdot d^{3/2}\cdot 2^{d/2}.$$
This was improved later by a breakthrough of Kabatjanski\u{\i} and Leven\v{s}te\u{\i}n~\cite{Kabat78}, using the linear programming method of Delsarte~\cite{Delsarte72}, to
$$K(d)\leq 2^{0.4041\ldots\cdot d}.$$

For the lower bound, Chabauti~\cite{Chabauty53}, Shannon~\cite{Shannon59}, and Wyner~\cite{Wyner65} independently observed that
\begin{equation}\label{eq-lower-bound-kissing-number}
	K(d)\geq(1+o(1))\frac{\sqrt{3\pi d}}{2\sqrt{2}}\Big(\frac{2}{\sqrt{3}}\Big)^d,
\end{equation}
 as every maximal arrangement has so many balls.
An improvement on the lower bound~\eqref{eq-lower-bound-kissing-number}, by a multiplicative factor $\Theta(d)$, was obtained by Jenssen, Joos and Perkins~\cite{JenssenJoosPerkins18} whose more general result (stated as~\cref{thm:sph-code-JJP} here) gives that
	\begin{equation}\label{eq:JJP}
	K(d)\geq (1+o(1))\frac{\sqrt{3\pi}}{2\sqrt{2}}\log\frac{3}{2\sqrt{2}}\cdot d^{3/2}\cdot \Big(\frac{2}{\sqrt{3}}\Big)^{d}.
	\end{equation}

We give a further constant factor improvement on the kissing numbers in high dimensions.
\begin{theorem}\label{thm:kissing}
		As $d\to\infty$, we have 
		$$K(d)\geq (1+o(1))\,{\frac{\sqrt{3\pi}}{4\sqrt2}}\,\log\frac{3}{2}\cdot d^{3/2}\cdot \Big(\frac{2}{\sqrt{3}}\Big)^{d}.$$
\end{theorem}

The leading constant ${\frac{\sqrt{3\pi}}{4\sqrt2}}\log\frac{3}{2}$ is about $0.2200...$, which is a factor of $3.442...$ improvement over the bound in~\eqref{eq:JJP}.

In fact, we can also improve lower bounds on the more general problem of the maximum size of a spherical code. 
A \emph{spherical code of angle $\theta$ in dimension $d$} is a set of vectors (also called \emph{codewords}) $x_1,\dots,x_k$ in the \emph{unit sphere}
$$
 \Sd:=\{x\in \mathbb R^d\mid \|x\|=1\},
$$
 such that $\langle x_i,x_j\rangle\leq\cos\theta$ for every $i\neq j$, that is, the angle between any two distinct vectors is at least $\theta$. The \emph{size} of such spherical code is $k$, the number of vectors. Let $A(d,\theta)$ denote the maximum size of a spherical code of angle $\theta$ in dimension~$d$.

Looking at the definition of a spherical code $x_1,\dots,x_k$ of angle $\theta$, we see that it corresponds to a set of non-overlapping caps $C_{\theta/2}(x_1),\dots,C_{\theta/2}(x_k)$, where
$$
C_{\theta}(x):=\{y\in \Sd: \langle x,y\rangle\geq\cos\theta\}
$$
denote the closed spherical cap of angular radius $\theta$ around $x\in\Sd$, Hence, determining $A(d,\theta)$ is equivalent to determining the maximum number of non-overlapping spherical caps of angular radius $\theta/2$ that can be packed in $\Sd$.
 
The kissing arrangement of unit spheres is a special case of spherical codes: if the centres of the kissing spheres are projected radially to the central unit sphere, then the obtained points form a spherical code of angle $\pi/3$ (and this transformation can be reversed). Thus $K(d)=A(d,\pi/3)$.

For $\theta\geq \pi/2$, Rankin~\cite{Rankin55} determined $A(d,\theta)$ exactly, so from now on we will assume that $\theta\in(0,\pi/2)$. For a measurable subset $A$ of $\Sd$, we write $s(A)$ for the normalised surface measure of $A$, i.e. $s(A):=\hat{s}(A)/\hat{s}(\Sd)$, where $\hat{s}(\cdot)$ is the usual surface measure (that is, the $(d-1)$-dimensional Hausdorff measure). Let us denote $s_d(\theta):=s(C_{\theta}(x))$.
For large $d$, the best known upper bound is by Kabatjanski\u{\i} and Leven\v{s}te\u{\i}n~\cite{Kabat78}
and states that
\begin{equation}\label{eq:KL}
A(d,\theta)\leq \me^{(1+o(1))\,\phi(\theta)\cdot d},
\end{equation}
for certain $\phi(\theta)>-\log \sin\theta$. 

The easy covering bound (observed by Chabauti~\cite{Chabauty53}, Shannon~\cite{Shannon59} and Wyner~\cite{Wyner65}) states that 
\begin{equation}\label{eq:CSW}
A(d,\theta)\geq s_d(\theta)^{-1}.
\end{equation}
Note that for fixed $\theta<\pi/2$, we have
\begin{equation}\label{eq:CapVol} 
 s_d(\theta)=(1+o(1))\,\frac{\sin^{d-1}\theta}{\sqrt{2\pi d}\,\cos(\theta)},\quad \mbox{as $d\to\infty$},
 \end{equation}
 and thus the bounds in~\eqref{eq:KL} and~\eqref{eq:CSW} are exponentially far apart from each other. The lower bound was improved by
 Jenssen, Joos and Perkins~\cite{JenssenJoosPerkins19}  by a linear factor in the dimension, showing that $A(d,\theta)=\Omega(d\cdot s_d(\theta)^{-1})$ as $d\to\infty$; see~\cref{thm:sph-code-JJP} here for the exact statement.

We also improve the lower bound on the maximum size of a spherical code by a constant factor (that depends on the angle $\theta\in (0,\pi/2)$) as $d\to\infty$:

\begin{theorem}\label{thm:sph-code}
	Let $\theta\in(0,\pi/2)$ be fixed. Then,
	$$A(d,\theta)\geq (1+o(1))\log\frac{\sin\theta}{\sqrt{2}\sin\frac{\theta}{2}}\cdot d\cdot s_d(\theta)^{-1}, \qquad \mbox{as } d\to\infty.$$
\end{theorem}

Observe that Theorem~\ref{thm:kissing} is obtained from \cref{thm:sph-code} by setting $\theta=\pi/3$ and using~\eqref{eq:CapVol}.

\medskip

We also look at sphere packings in the Euclidean space $\bR^d$ of maximum density.
 Given a radius $r>0$,
a \emph{(sphere) packing} $X$ is a subset of $\mathbb{R}^d$ such that
every two distinct elements of $X$ are at distance at least $2r$ (or, equivalently, if the radius-$r$ balls centred at $X$ are non-overlapping). 
The \emph{sphere packing density} 
is defined by
\begin{equation}\label{eq:density}
\theta(d):=\sup_{\mathrm{packing}\ {X}}\ \limsup_{R\to\infty}\frac{\mathrm{vol}(B_R(0)\cap (\cup_{x\in X} B_r(x))}{\mathrm{vol}(B_R(0))},
\end{equation}
 where $B_R(x)$ denotes the closed ball of radius $R>0$ centred at $x\in\mathbb R^d$ and $0\in\mathbb R^d$ is the origin. In other words, we try to cover asymptotically as large as possible fraction of the volume of a growing ball in~$\mathbb R^d$ by non-overlapping balls of fixed radius~$r$.

It is clear that $\theta(1)=1$. Thue~\cite{Thue1892}, in 1892, proved that $\theta(2)=\pi/\sqrt{12}=0.9068\dots$\,, which is achieved by the hexagonal lattice; Hales~\cite{Hales2005} proved in 2005 that $\theta(3)=\pi/\sqrt{18}=0.7404...$\,. The cases $d=8$ and $d=24$ have been recently resolved due to the work of Viazovska~\cite{Viazovska2017} ($d=8$) and Cohn, Kumar, Miller, Radchenko and Viazovska~\cite{Cohnetal2017} ($d=24$). 

However, the value of $\theta(d)$ is unknown for any other $d$. When $d\to\infty$, there are some upper and lower bounds for $\theta(d)$, but they are exponentially far apart. The best known upper bound $\theta(d)\leq 2^{-(0.5990...+o(1))\cdot d}$ is due to  Kabatjanski\u{\i} and Leven\v{s}te\u{\i}n~\cite{Kabat78}, obtained by applying their bounds on spherical codes. As shown by Cohn and Zhao~\cite{CohnZhao2014}, the more direct approach of Cohn and Elkies~\cite{CohnElkies03} gives at least as strong upper bound on $\theta(d)$ as that of~\cite{Kabat78}.

The trivial lower bound $\theta(d)\geq 2^{-d}$ (take a maximal sphere packing and observe that balls of doubled size cover the whole space) was improved by a factor of $d$ by Rogers~\cite{Rogers47} by analysing a random lattice packing of $\bR^d$ by making use of the Siegel mean-value theorem. Later, there have been several subsequent improvements to the constant by Davenport and Rogers~\cite{DavRogers47}, Ball~\cite{Ball92}, Vance~\cite{Vance2011}. The current best bound is by Venkatesh~\cite{Venkatesh2013}, who proved a general bound $\theta(d)\geq (65963+o_d(1))\,d\cdot 2^{-d}$, and for a sparse sequence  of dimensions $\{d_i\}_{i\in\bN}$, a bound $\theta(d_i)=\Omega(d_i\cdot\log\log d_i\cdot 2^{-d_i})$. (In contrast, the best known lower bound on the kissing number for $d\to\infty$ coming from a lattice is exponentially in $d$ smaller than the easy bound in~\eqref{eq-lower-bound-kissing-number}, see Vl\u{a}du\textcommabelow{t}~\cite{Vladut2019}.)

One can try to find lower bounds by taking a random packing $X$ inside a bounded measurable set $S\subseteq\mathbb R^d$. Jenssen, Joos and Perkins~\cite{JenssenJoosPerkins19} investigated the \emph{hard sphere model of fugacity $\lambda$} where we take the Poisson point process $X$ on $S$ of intensity $\lambda$ conditioned on being a packing, with the radius chosen so that the balls we pack have volume $1$. 
Define the \emph{expected packing density} $\alpha_S(\lambda):=\mathbb E[\,|X|\,]/\vol(S)$ and
observe that, for any $\lambda>0$, the limit superior of $\alpha_S(\lambda)$ when $S\subseteq\mathbb R^d$ is a large ball  is a lower bound on $\theta(d)$. Jenssen, Joos and Perkins~\cite{JenssenJoosPerkins19} were able to prove that $\theta(d)\ge (\log(2/\sqrt{3})+o(1))\, d\cdot 2^{-d}$ via this method (see \cref{thm-bound-exp-packing-density} here). This does not improve the lower bound of Venkatesh~\cite{Venkatesh2013} but the full potential of this approach is unclear.
(In fact, this approach is used in~\cite{JenssenJoosPerkins18} and in this paper to improve the best known lower bounds on spherical codes and kissing numbers, as stated above.)

Our next result improves the lower bound on $\theta(d)$
given by this method by a multiplicative constant of $\log2/\log({4}/{3})+o(1)=2.409...$ as $d\to\infty$.

\begin{theorem}\label{thm-improvement-sph-pack-dens}
	 For every $\eps>0$, there are $\delta>0$ and $d_0$ such that if $d\ge d_0$, $\lambda\ge (1/\sqrt{2}-\delta)^d$ and $S\subseteq\mathbb{R}^d$ is any bounded and measurable set with positive measure, then 
$$
	\alpha_S(\lambda)\geq(\log\sqrt{2}-\eps)\cdot d\cdot 2^{-d}.$$
\end{theorem}

Our approach builds on the work of Jenssen, Joos and Perkins~\cite{JenssenJoosPerkins18,JenssenJoosPerkins19}. As in these papers, our lower bounds are obtained by analysing  the structure of the random packing $X$ around a uniformly chosen random point $v\in S$. In brief, the new ideas that lead to our improvements are lower bounding the expected size of the random configuration around $v$ in a more direct way (via Lemma~\ref{lm:pois-mean-sph}) and using the known re-arrangement 
inequalities for $\Sd$ and $\mathbb R^d$. Since the reader (like us) may find Euclidean geometry more intuitive than the spherical one, we first present a rather detailed proof for sphere packing in the Euclidean space that introduces the same new ideas as the case of spherical codes.


\section{Notation}\label{sec-prelim}

For $n\in\mathbb{N}$, write $[n]:=\{1,\dots,n\}$. If we claim that a result holds e.g.\ for $0<a\ll b,c\ll d$, it means that there exist positive functions $f$ and $g$ such that the result holds as long as $a<f(b,c)$ and $b<g(d)$ and $c<g(d)$. We will not compute these functions explicitly.

We will sometimes use a standard abuse of notation by denoting  all the probability measures that we use as $\bP$, even if they refer to different probability spaces, but it will be clear from the context with respect to which one it is used; and similarly for the expectation, which will be denoted by~$\bE$. 
All the logarithms will be in base $\me$. Moreover, we will use the standard asymptotic notation for non-negative functions $f$ and $g$: $f(d)=O(g)$ means that when $d$ tends to infinity $f(d)/g(d)$ is bounded by a constant independent of $d$; $f(d)=o(g)$ means that when $d$ is large~$f(d)/g(d)$ tends to zero; $f(d)=\Omega(g(d))$ means that, there exist constants $C$ and $d_0$ such that, for every $d>d_0$, $f(d)\geq C\cdot g(d)$; and $f(d)=\Theta(g(d))$ means that, there exist constants $C_1,C_2,d_0$ such that, for every $d>d_0$, it holds that $C_1\cdot g(d)\leq f(d)\leq C_2\cdot g(d)$.

Let $\vol$ denote the Lebesgue measure in $\mathbb{R}^d$.
Let $B^\circ_r(x):=\{y\in \mathbb R^d\mid d(x,y)<r\}$ denote the open radius-$r$ ball centred at $x$ (and recall that $B_r(x)$ denotes the closed ball).
Likewise,
$$
C^\circ_{\theta}(x):=\{y\in \Sd: \langle x,y\rangle>\cos\theta\}
$$
denotes the open spherical cap of angular radius $\theta$ around $x$ in
the sphere $\Sd$.

Also, we will be using, without any further mention, Fubini-Tonelli's Theorem (see e.g.~\cite[Section 2.3]{RealAnSteinShakarchi}) which states that measurable non-negative functions can be integrated in any order of variables and, if we integrate out any subset of variables then the obtained function in the remaining ones is defined almost everywhere and is measurable.



\section{Sphere packing in Euclidean space}
\label{sec-improve-hard-sph}
Before proving Theorem~\ref{thm-improvement-sph-pack-dens} we need to discuss the hard sphere model in some detail.

\subsection{The hard sphere model}\label{sec-hard-sphere-model}

Recall that a (sphere) packing ${X}$ is a subset of $\mathbb{R}^d$ consisting of centres of balls of equal radii with disjoint interiors. Since the density of a packing, as defined in~\eqref{eq:density}, will not change if we scale the whole picture by a constant factor, we assume from now on that the balls associated to a sphere packing $X$ are all of volume one, and we write $r_d$ for the radius of a ball of volume one in~$\mathbb{R}^d$. This will allow us to treat the number $|X|$ of points in $X$ as the total volume of the balls.

Let $S\subseteq \mathbb{R}^d$ be a bounded measurable set.
The \emph{hard sphere model on $S$} is a probability distribution over sphere packings $X\subseteq S$. Before giving a formal definition, we provide some intuition. Consider an infinite graph with vertex set $S$ in which two points are neighbours if they are of distance less than $2r_d$ apart. Then, given a sphere packing $X$, we can think of the the set $X$ of points as an independent set in our graph, and the hard sphere model can be thought of as a continuous version of the so-called \emph{hard core model} that samples independent sets in graphs (some recent results involving the latter can be found in e.g.\ \cite{DavJenPerkRob17,DavJenPerkRob18,PerarnauPerkins18}).

There are two versions of the hard sphere model: the \emph{canonical ensemble}, which is a uniformly chosen random packing of a given fixed density (i.e. the number of balls is fixed); and the \emph{grand canonical ensemble}, which is a random packing with variable density determined by a fugacity parameter $\lambda>0$. More precisely, let
$S\subseteq\mathbb{R}^d$ be a bounded, measurable set of positive measure and let $k$ be a non-negative integer. Define
$$P_k(S):=\{\{x_1,\dots,x_k\}\mid x_1,\dots,x_k\in S,\ \forall\, 1\le i<j\le k\ d(x_i,x_j)\ge 2r_d\}$$
to be the sets of sphere packings of size $k$ with centres inside~$S$. Note that we allow centres arbitrarily close to the boundary of $S$, that is, we do not require that the whole ball stays inside $S$, but only its centre. Then in the 
\emph{canonical hard sphere model on} $S$, we take a $k$-tuple $X_k\in P_k(S)$ uniformly at random (if the measure of $P_k(S)$ is positive). The \emph{partition function} of the canonical ensemble is given by
$$\hat{Z}_S(k):=\frac{1}{k!}\int_{S^k}\mathds{1}_{\mathcal{D}(x_1,\dots, x_k)}\dd x_1\cdots \dd x_k,\quad k\ge 1,$$
where $\mathcal{D}(x_1,\dots,x_k)$ is the event that $d(x_i,x_j)\ge 2r_d$ for every distinct $i,j\in[k]$. We also define $\hat{Z}_S(0):=1$. Note that $\hat Z_S(k)=0$ if the measure of $P_S(k)$ is zero.  Observe that $\hat{Z}_S(k)$ is the measure of the set of the legitimate size-$k$ packings $P_k(S)$ and the probability that the random uniform $k$-tuple $X_k\in S^k$ is in~$P_k(S)$ is
$$\bP[\,X_k\in P_k(S)\,]=\frac{k!}{\mathrm{vol}(S)^k}\,\hat{Z}_S(k).$$
Note that the canonical ensemble is the analogue of the uniform distribution in the family of independent sets of fixed size of a graph.

\smallskip

In the \emph{grand canonical hard sphere model on} a bounded measurable subset $S\subseteq \mathbb R^d$ \emph{at fugacity} $\lambda>0$, a random set $X$ of unordered points is sampled according to a Poisson point process on $S$ of intensity~$\lambda$, conditioned on the event that $d(x,y)\geq 2r_d$ for every distinct $x,y\in X$. Note that we condition on the event of positive measure: for example, the unconditioned Poisson set is empty with probability $\me^{-\lambda\cdot\vol(S)}>0$. We will write $\mu_{S,\lambda}$ for the probability measure of the hard sphere model at fugacity~$\lambda$ on $S$, and we may abbreviate $\mathbb P_{X\sim \mu_{S,\lambda}}$ to $\mathbb P_X$, or even to $\mathbb P$ when the meaning is clear. 

Let us present an equivalent description of the same distribution. 
Define the \emph{partition function}
$$Z_S(\lambda):=\sum_{k=0}^\infty\frac{\lambda^k}{k!}\int_{S^k}\mathds{1}_{\mathcal{D}(x_1,\dots,x_k)}\dd x_1\cdots \dd x_k=\sum_{k=0}^\infty\lambda^k\hat{Z}_S(k).$$
Note that if $S$ is bounded, then $\hat{Z}_S(k)=0$ for large $k$ (so in particular, we do not need to worry about the convergence of the sum). This gives a random set $X\subseteq S$ as follows: first, we choose a non-negative integer $k$ at random with probability proportional to $\lambda^k\hat{Z}_S(k)$, and then we choose a $k$-tuple $X\in P_k(S)$ from the canonical hard sphere model. Let us show that these two distributions are indeed equal.


\begin{lemma}\label{lm:equiv} Let $S\subseteq \mathbb R^d$ be a bounded measurable set of positive measure and let $\lambda>0$. 
	Let~$X$ be the Poisson process of intensity $\lambda$ on $S$ conditioned on being a packing. For an integer $k\ge 0$, let~$E_k$ be the event that $|X|=k$. Then $\mathbb P(E_k)=\lambda^k \hat Z_S(k)/Z_S(\lambda)$. Moreover, if $E_k$ has positive measure then $X$ conditioned on $E_k$ is uniformly distributed in~$P_k(S)$.
\end{lemma}

\begin{proof}
	The probability that the unconditioned Poisson process $Y$ of intensity $\lambda$ on $S$ has exactly $k$ points is
	$$\me^{-\lambda\cdot \vol(S)}\frac{(\lambda\cdot\vol(S))^k}{k!}.$$
	Conditioned on $|Y|=k$,  we have that $Y$ is a uniformly chosen random element of $S^k$. Thus the probability that $Y$ is a packing of size $k$ is
	$$
	p_k:=\frac{\me^{-\lambda\cdot\vol(S)}(\lambda\cdot\vol(S))^k}{k!} 
	\cdot \frac1{\vol(S)^k}\int_{S^k} \mathds 1_{\mathcal D(x_1,\dots,x_k)}\dd x_1\ldots\dd x_k
	=\me^{-\lambda\cdot\vol(S)}\, \lambda^k \hat Z_S(k).$$
	Therefore, the probability that $|X|=k$ is
	$$
	\mathbb P[\,E_k\,]=\frac{p_k}{\sum_{i=0}^\infty p_i}
	=\frac{\lambda^k \hat Z_S(k)}{Z_S(\lambda)},
	$$
	proving the first claim. The second claim follows from the uniformity of $Y$ in $S^k$ when conditioned on its size~$k$.
\end{proof}

One of the main properties of the hard sphere model is the spatial Markov property.  

\begin{lemma}[Spatial Markov Property]\label{lm:SMP} Let $A\subseteq S\subseteq \mathbb R^d$ be bounded measurable sets of positive measure. Let $\lambda>0$ and let $X\sim \mu_{S,\lambda}$.
	Let $Y$ be obtained from $X$ by removing $X\cap A$ and adding the points produced by the hard sphere fugacity-$\lambda$ process on 
	\begin{equation}\label{eq:TA}
	\mathbf{T}_A(X):=\{x\in A:\forall y\in X\setminus A\ \ d(x,y)\ge 2r_d \}.
	\end{equation} 
	Then the distributions of the point processes $X$ and $Y$ are the same.
\end{lemma}

\begin{proof} 
	Let us show first that for any non-negative integers $k$ and $\ell$ it holds that 
	\begin{equation}\label{eq:star}
		\bP[\,|X\cap A|=k,\,|X\setminus A|=\ell\,]=\bP[\,|Y\cap A|=k,\,|X\setminus A|=\ell\,].
	\end{equation}

	Take any integers $k,\ell\ge 0$. Let $E_{k,\ell}$ denote the event that $|X\cap A|=k$ and $|X\setminus A|=\ell$. Write $\mathbf{x}:=(x_1,\dots,x_\ell)$, $\mathbf{x}':=(x_{\ell+1},\dots,x_{\ell+k})$ and $(\mathbf{x},\mathbf{x}'):=(x_1,\dots,x_\ell,x_{\ell+1},\dots,x_{\ell+k})$.
	We have by \cref{lm:equiv} that
	\begin{eqnarray}
		\mathbb P[\,E_{k,\ell}\,] 
		&=& \frac{1}{Z_S(\lambda)}\frac{\lambda^{k+\ell}}{(k+\ell)!}\cdot 
		\int_{S^{k+\ell}}\mathds{1}_{\mathcal{D}(\mathbf{x},\mathbf{x}')}\mathds{1}_{|\{x_1,\dots,x_{k+\ell}\}\cap A|=k}\dd \mathbf{x} \dd \mathbf{x}'\nonumber\\
		&=& \frac{1}{Z_S(\lambda)}\frac{\lambda^{k+\ell}}{(k+\ell)!}\cdot {k+\ell\choose k} 
		\int_{S^{k+\ell}}\mathds{1}_{\mathcal{D}(\mathbf{x},\mathbf{x}')}\,
		\mathds{1}_{\{\mathbf{x}\in (S\setminus A)^\ell\}}\,
		\mathds{1}_{\{\mathbf{x}'\in A^k\}}\,
		\dd \mathbf{x} \dd \mathbf{x}'\nonumber\\
		&=&\frac{1}{Z_S(\lambda)}\frac{\lambda^{k+\ell}}{k!\ell!}\int_{(S\setminus A)^\ell}\mathds{1}_{\mathcal{D}(\mathbf{x})}\Bigg(\int_{\mathbf{T}_A^k(\mathbf x)}\mathds{1}_{\mathcal{D}(\mathbf{x}')}\dd \mathbf{x}'\Bigg)\dd \mathbf{x}.\label{eq-LHS}
	\end{eqnarray}
	Here, the first equality is trivially true if the probability of $|X|=k+\ell$ is zero; otherwise, it is a consequence of
	$$\mathbb P[\,E_{k,\ell}\,] = \mathbb P[\,|X|=k+\ell]\cdot \mathbb P[\,E_{k,\ell}\,|\,|X|=k+\ell\,].
	$$

	On the other hand, for any $k,\ell,j\ge 0$, let 
	$$\mathrm{P}_{k,\ell,j}:=\bP[\,|X\setminus A|=\ell,\,|Y\cap A|=k, \,|X\cap A|=j\,].
	$$
	Clearly, the right-hand side of~\eqref{eq:star} is $\sum_{j=0}^\infty \mathrm{P}_{k,\ell,j}$. This sum, using similar arguments as before and denoting $\mathbf{x}_j:=(x_{\ell+1},\dots,x_{\ell+j})$ and $\mathbf{y}:=(y_{1},\dots,y_{k})$,
	can be re-written as
	\begin{align}
			\sum_{j=0}^\infty \mathrm{P}_{k,\ell,j}&= \sum_{j=0}^\infty \frac{1}{Z_S(\lambda)}\frac{\lambda^{j+\ell}}{(j+\ell)!}{j+\ell \choose j}\cdot \int_{(S\setminus A)^\ell}\mathds{1}_{\mathcal{D}(\mathbf{x})}\Bigg(\int_{\mathbf{T}_A^j(\mathbf x)}\mathds{1}_{\mathcal{D}(\mathbf{x}_j)}\dd \mathbf{x}_j\Bigg)
			\Bigg(\frac{\lambda^k\hat Z_{\mathbf{T}_A(\mathbf x)}(k)}{Z_{\mathbf{T}_A(\mathbf x)}(\lambda)}
			\Bigg)\dd \mathbf{x}\nonumber\\
			&=\frac{1}{Z_S(\lambda)}\frac{\lambda^{k+\ell}}{k!\ell!} 
			\int_{(S\setminus A)^\ell}\mathds{1}_{\mathcal{D}(\mathbf{x})}
			\Bigg(\sum_{j=0}^\infty \frac{\lambda^j\hat Z_{\mathbf{T}_A(\mathbf x)}(j)}{Z_{\mathbf{T}_A(\mathbf x)}(\lambda)}\Bigg)
			\Bigg(\int_{\mathbf{T}_A^k(\mathbf x)}\mathds{1}_{\mathcal{D}(\mathbf{y})}\dd \mathbf{y}\Bigg)\dd \mathbf{x}\nonumber\\
			&=\frac{1}{Z_S(\lambda)}\frac{\lambda^{k+\ell}}{k!\ell!} 
			\int_{(S\setminus A)^\ell}\mathds{1}_{\mathcal{D}(\mathbf{x})}
			\Bigg(\int_{\mathbf{T}_A^k(\mathbf x)}\mathds{1}_{\mathcal{D}(\mathbf{y})}\dd \mathbf{y}\Bigg)\dd \mathbf{x}.\label{eq:Y}
	\end{align}
	
	For every pair $(k,\ell)$ with the two (equal) probabilities in~\eqref{eq:star} non-zero, if we condition on the event $|X\cap A|=k$ and $|X\setminus A|=\ell$ (resp.\ on the event $|Y\cap A|=k$ and $|Y\setminus A|=\ell$), then each of $X$ and $Y$, as a random point of $S^{k+\ell}$, has the same density function with respect to the Lebesgue measure by the calculations in~\eqref{eq-LHS} and~\eqref{eq:Y}.
 Of course, when we ignore the pairs $(k,\ell)$ where the probabilities in~\eqref{eq:star} are zero, we ignore, by countable additivity, a set of measure 0 (which does not affect our distributions). Thus the distributions of $X$ and $Y$ are the same, as desired.
\end{proof}

In other words, the process of generating $Y$ in Lemma~\ref{lm:SMP} gives a regular conditional distribution of $X$ with respect to the $\sigma$-algebra of $X\setminus A$, that is, the $\sigma$-algebra generated by sets of the form 
$$
\{X\mid \mbox{$X\subseteq S$ is a packing such that } \forall i\in [m]\ |X\cap B_i|=k_i\},
$$
 for integers $m,k_1,\dots,k_m\ge 0$ and measurable subsets $B_1,\dots,B_m$ of $S\setminus A$.
This gives a well-defined meaning to phrases like \emph{``the distribution of $X\cap A$ conditioned on $X\setminus A=\{x_1,\dots,x_\ell\}$''} by which we will mean $\mu_{\mathbf T_{A}(x_1,\dots,x_\ell),\lambda}$, that is, the hard sphere distribution on $\mathbf{T}_{A}(x_1,\dots,x_\ell)$ of the same fugacity~$\lambda$. 

The \emph{expected packing density}, $\alpha_S(\lambda)$, of the (grand canonical) hard sphere model is the expected total volume of the balls in the random packing normalised by the measure of $S$. As we consider balls of unit volume, this is the number of centres in $S$ normalised by the measure of $S$, that is,
	$$\alpha_S(\lambda):=\frac{\bE_{X\sim \mu_{S,\lambda}}[\,|X|\,]}{\mathrm{vol}(S)}.$$

Jenssen, Joos and Perkins in~\cite{JenssenJoosPerkins19} proved the following asymptotic lower bound for the expected packing density.

\begin{theorem}[Theorem 2, \cite{JenssenJoosPerkins19}]\label{thm-bound-exp-packing-density}
	Let $d\to\infty$, let $S\subseteq \mathbb{R}^d$ be a bounded and measurable set of positive measure and let $\lambda \geq 3^{-d/2}$ be arbitrary. Then
	$$\alpha_S(\lambda)\geq (1+o(1))\log\frac{2}{\sqrt{3}}\cdot d\cdot 2^{-d}.$$
\end{theorem}

Our Theorem~\ref{thm-improvement-sph-pack-dens} gives a constant factor improvement over it (although it applies only to higher values of $\lambda$ than \cref{thm-bound-exp-packing-density}). Both results give a lower bound on the sphere packing density via $\theta(d)\geq\limsup_{n\to\infty} \alpha_{B_n(0)}(\lambda)$ (see e.g.~\cite[Lemma 1]{JenssenJoosPerkins19}).

We need the following auxiliary results from~\cite{JenssenJoosPerkins19} in order to prove~Theorem~\ref{thm-improvement-sph-pack-dens}. For completeness, we include their short proofs.

\begin{lemma}\label{lem-interm-results-pack-dens-JJP}
	Let $S\subseteq\mathbb{R}^d$ be a bounded, measurable set of positive measure, $\lambda>0$ and $X\sim\mu_{S,\lambda}$. Then, the followings hold:
	\begin{itemize}
	 \item [$(\mathrm{i})$] $\alpha_S(\lambda)=\frac{\lambda}{\mathrm{vol}(S)}\,\big(\log Z_S(\lambda)\big)'$;
	 \item[$(\mathrm{ii})$] $\alpha_S(\lambda)$ is a strictly increasing function of $\lambda$;
	\item[$(\mathrm{iii})$] $Z_S(\lambda)\leq \me^{\lambda\cdot\vol(S)}.$ \label{eq:small-t}
	\end{itemize}
\end{lemma}
\begin{proof}
To see $(\mathrm{i})$, we compute:
	\begin{align*}
		\alpha_S(\lambda)&=\frac{1}{\mathrm{vol}(S)}\sum_{k=1}^{\infty} k\cdot\bP[\,|X|=k\,]=\frac{1}{\mathrm{vol}(S)}\sum_{k=1}^{\infty} k\cdot\frac{\lambda^k\hat{Z}_S(k)}{Z_S(\lambda)}\\
		&=\frac{\lambda}{\mathrm{vol}(S)}\,\frac{Z'_S(\lambda)}{Z_S(\lambda)}=\frac{\lambda}{\mathrm{vol}(S)}\,\big(\log Z_S(\lambda)\big)'.
	\end{align*}
	We can get $(\mathrm{ii})$ by differentiating with respect to $\lambda$ the expression given by $(\mathrm{i})$:
	\begin{equation*}
		\begin{split}
			\lambda\cdot\mathrm{vol}(S)\cdot\alpha'_S(\lambda)&=\lambda\cdot\mathrm{vol}(S)\Big(\frac{1}{\mathrm{vol}(S)}\big(\log Z_S(\lambda)\big)'+\frac{\lambda}{\mathrm{vol}(S)}\big(\log Z_S(\lambda)\big)''\Big)\\
			&=\frac{\lambda^2 Z''_S(\lambda)}{Z_S(\lambda)}-
			\left(\frac{\lambda Z'_S(\lambda)}{Z^2_S(\lambda)}\right)^2+\frac{\lambda Z'_S(\lambda)}{Z_S(\lambda)}
			\\
			&=\bE\big[\,|X|\,(|X|-1)\,\big]-\big(\bE[\,|X|\,]\big)^2+\bE[\,|X|\,]=\rV[\,|X|\,]>0,
		\end{split}
	\end{equation*}
 where the identity $\bE\big[\,|X|\,(|X|-1)\,\big]={\lambda^2 Z''_S(\lambda)}/{Z_S(\lambda)}$ can be proved very similarly as the identity in~$(\mathrm{i})$.

To see $(\mathrm{iii})$, note that
$$Z_S(\lambda)=\sum_{k=0}^\infty\lambda^k\hat{Z}_S(k)=\sum_{k=0}^\infty\frac{\lambda^k}{k!}\int_{S^k}\mathds{1}_{\mathcal{D}(x_1,\dots,x_k)}\dd x_1\cdots \dd x_k\leq\sum_{k=0}^\infty\frac{\lambda^k}{k!}\mathrm{vol}(S)^k=\me^{\lambda\cdot\mathrm{vol}(S)}.$$

\end{proof}

\subsection{Externally uncovered neighbourhood $\mathbf{T}$}

Now we do a two-part experiment: let $X$ be a random configuration of centres drawn according to the hard sphere model on $S$ at fugacity $\lambda$ and, independently, choose a point $v$ uniformly at random from $S$. Define the set
\begin{equation}\label{defn:T}
	\mathbf{T}:=\mathbf{T}(X,v)=\{x\in B^\circ_{2r_d}(v)\cap S: \forall y\in X\setminus B^\circ_{2r_d}(v)\ \ d(x,y)\ge 2r_d \},
\end{equation}
which is the set of all points of $S$ in the open ball of radius $2r_d$ around $v$ that are suitable to be a centre of a new ball to add to the packing $X\setminus B^\circ_{2r_d}(v)$ (because they are not blocked by a centre outside $B^\circ_{2r_d}(v)$). See Figure~\ref{fig:setT}. This is the same definition as when we take $A:= B^\circ_{2r_d}(v)\cap S$ in~\cref{eq:TA}. The set~$\mathbf{T}$ is called the set of \emph{externally uncovered points} in the neighbourhood of $v$ (with respect to $X$). Note that $\vol(\mathbf{T})>0$ almost surely.

\begin{figure}[h]
	\begin{center}
\definecolor{cadetgrey}{rgb}{0.57, 0.64, 0.69}
\definecolor{darkgray}{rgb}{0.66, 0.66, 0.66}
	\begin{tikzpicture}
		\draw[black, thick] (0,0) rectangle (7,4);
		\draw[darkgray] (3,1) circle [radius=12pt];
		\draw[darkgray] (5.7,1.3) circle [radius=12pt];
		\draw[darkgray] (3,3.1) circle [radius=12pt];
		\draw[darkgray] (1,3) circle [radius=12pt];
		\draw[darkgray] (5.2,3) circle [radius=12pt];
		
		\draw[darkgray, dashed] (3,1) circle [radius=24pt];
		\draw[darkgray, dashed] (5.7,1.3) circle [radius=24pt];
		\draw[darkgray, dashed] (3,3.1) circle [radius=24pt];
		\draw[white, fill] (3,1) circle [radius=23.7pt];
		\draw[white, fill] (5.7,1.3) circle [radius=23.7pt];
		\draw[darkgray] (3,1) circle [radius=12pt];
		\draw[darkgray] (5.7,1.3) circle [radius=12pt];
		\node[inner sep= 1pt](a1) at (3,1)[circle,fill]{};
		\node[inner sep= 1pt](a1) at (3,3.1)[circle,fill]{};
		\node[inner sep= 1pt](a1) at (5.7,1.3)[circle,fill]{};
		\node[inner sep= 1pt](a2) at (1,3)[circle,fill]{};
		\node[inner sep= 1pt](a1) at (5.2,3)[circle,fill]{};
		\draw[->] (a2) -- (0.6,3);
		\node[inner sep= 1pt](v) at (4.3,1.5)[circle,fill]{};
		\draw[->] (v) -- (3.5,1.5);
		\draw[cadetgrey, dashed] (4.3,1.5) circle [radius=24pt];
		\begin{pgfonlayer}{background}
			\draw[cadetgrey, dashed, fill] (4.3,1.5) circle [radius=23.8pt];
		\end{pgfonlayer}
		\node[inner sep= 1pt](l1) at (5.4,3){\tiny$X$};
		\node[inner sep= 1pt](l1) at (4.5,1.5){\small$v$};
		\node[inner sep= 1pt](l1) at (4,1.65){\tiny$2r_d$};
		\node[inner sep= 1pt](l1) at (0.8,3.1){\tiny$r_d$};
		\node[inner sep= 1pt](l1) at (4.5,2){\small$T$};
	\end{tikzpicture}
	\caption{The set $\mathbf{T}$ of externally uncovered points.\label{fig:setT}}
\end{center}
\end{figure}

We will also need the following two results from~\cite{JenssenJoosPerkins19} that relate this two-part experiment to $\alpha_S(\lambda)$, and we include their proofs for the reader's convenience.

\begin{lemma}\label{lem-interm-results-pack-dens-JJP2}
	Let $S\subseteq\mathbb{R}^d$ be a bounded, measurable set of positive measure, $\lambda>0$, $X\sim\mu_{S,\lambda}$ and let $v\in S$ be a random point chosen uniformly from $S$, independent of $X$. Let $\mathbf{T}=\mathbf{T}(X,v)$ be as in~\cref{defn:T}. Then, the following statements hold:
	\begin{itemize}
		\item[$(\mathrm{i})$] $\alpha_S(\lambda)=\lambda\cdot\bE_{X,v}\Big[\frac{1}{Z_{\mathbf{T}}(\lambda)}\Big]$;
		\item[$(\mathrm{ii})$] $\alpha_S(\lambda)\geq 2^{-d}\cdot\bE_{X,v}\big[\alpha_{\mathbf{T}}(\lambda)\cdot\mathrm{vol}(\mathbf{T})\big].$ \label{eq:large-t}
	\end{itemize}
\end{lemma}

\begin{proof}
	For $(\mathrm{i})$, recall that $\mathcal{D}(x_0,x_1,\dots, x_k)$ is the event that every two of the points $x_0,\dots,x_k$ are at distance at least $2r_d$. Then, we compute
	\begin{equation*}
		\begin{split}
			\alpha_S(\lambda)&=\frac{\bE[\,|X|\,]}{\mathrm{vol}(S)}=\frac{1}{\mathrm{vol}(S)}\sum_{k=0}^\infty(k+1)\cdot\bP[\,|X|=k+1\,]\\
			&=\frac{1}{\mathrm{vol}(S)Z_S(\lambda)}\sum_{k=0}^\infty\int_{S^{k+1}}\frac{\lambda^{k+1}}{k!}\mathds{1}_{\mathcal{D}(x_0,x_1,\dots, x_k)}\dd x_1\cdots \dd x_k\dd x_0\\
			&=\frac{\lambda}{\mathrm{vol}(S)}\int_S\frac{1}{Z_S(\lambda)}\cdot\bigg(1+\sum_{k=1}^\infty\int_{S^{k}}\frac{\lambda^{k}}{k!}\mathds{1}_{\mathcal{D}(x_0,x_1,\dots, x_k)}\dd x_1\cdots \dd x_k\bigg)\dd x_0\\
			&=\frac{\lambda}{\mathrm{vol}(S)}\int_S\bP_X[\,d(x_0,X)\ge 2r_d\,]\dd x_0\\
			&=\lambda\cdot\bE_{X,v}\big[\,\mathds{1}_{\{\mathbf{T}\cap X=\varnothing\}}\,\big]=\lambda\cdot\bE_{X,v}\Big[\frac{1}{Z_{\mathbf{T}}(\lambda)}\Big],
		\end{split}
	\end{equation*}
	where the last equality follows from Lemma~\ref{lm:SMP}, the spatial Markov property, applied to $A:=B^\circ_{2r_d}(v)\cap S$.
	
	For $(\mathrm{ii})$, as $\mathrm{vol}(S\cap B_{2r_d}(u))\leq 2^d$ for any $u\in S$, we have
	\begin{equation*}
		\alpha_S(\lambda)=\frac{\bE_X[\,|X|\,]}{\mathrm{vol}(S)}\geq2^{-d}\cdot\bE_{X,v}\big[\,|X\cap B^\circ_{2r_d}(v)|\,\big]=2^{-d}\cdot\bE_{X,v}[\,\alpha_{\mathbf{T}}(\lambda)\cdot \mathrm{vol}(\mathbf{T})\,],
	\end{equation*}
	where the last equality follows again from the spatial Markov property.
\end{proof}


\subsection{Local analysis in $\mathbf{T}$}\label{se:local}

As in~\cite{JenssenJoosPerkins19}, the key part of our argument is a local analysis of the number of centres inside the externally uncovered set $\mathbf{T}=\mathbf{T}(X,v)$. However, our proof deviates from~\cite{JenssenJoosPerkins19} from this point on. Roughly speaking, we write our lower bound on $\theta(d)$ in terms of the distribution of $t:=\vol(\mathbf{T}(X,v))$, instead of $\mathbb E_{X,v}[\, \log Z_{\mathbf{T}(X,v)}\,]$ as was done in~\cite{JenssenJoosPerkins19} and observe, using a standard rearrangement inequality, that a worst case for our bound is when $t=(\log \sqrt{2}+o(1))\,d/\lambda$ is constant and $\mathbf{T}(X,v)$ is a ball. Thus our improvement comes by adding new geometrical considerations into the proof.


For the proof we will need some auxiliary results from Real Analysis. We say that a measurable function $f:\mathbb{R}^d\to [0,\infty)$ \emph{vanishes at infinity} if for every $t>0$ the level set $\{x\in\mathbb{R}^d: f(x)> t\}$ has finite Lebesgue measure.
For a measurable bounded set $A\subseteq \mathbb{R}^d$
and a measurable function $f:\mathbb{R}^d\to [0,\infty)$ that vanishes at infinity, their \emph{symmetric rearrangements} are, respectively
$$
A^*:=B_{\vol(A)^{1/d}\cdot r_d}(0),
$$
the ball centred at $0$ of the same measure as $A$, and
$$
f^*(x):=\int_0^\infty \mathds{1}_{\{y:f(y)>t\}^*}(x)\,\mathrm{d} t,\quad x\in\mathbb{R}^d,
$$
the radially decreasing symmetric function with the same measures of the level sets as $f$.
For more details, see e.g.~\cite[Chapter 3.3]{LiebLoss2001}.

For a measurable bounded $T\subseteq \mathbb{R}^d$ define
\begin{equation}\label{defn:f}
f(T):=\int_T \vol(B_{2r_d}(u)\cap T)\,\mathrm{d} u.
\end{equation}
Note that, if $\vol(T)>0$, then $f(T)/\vol(T)$ is the expected measure of the intersection $B_{2r_d}(u)\cap T$ for a random point $u$ uniformly chosen from $T$.

We will make use of the following result, which says that the function $f$ defined above is maximised by a ball of the same measure as $T$.

\begin{lemma}\label{lem-ineq-f} Let $f$ be as in~\cref{defn:f}. For every bounded measurable $T\subseteq \mathbb{R}^d$, we have
	\begin{equation*}
		f(T)\le f(T^*).
	\end{equation*}	
\end{lemma}
\begin{proof}
	This is an immediate consequence of Riesz's rearrangement inequality (for a modern exposition, see e.g.~\cite[Theorem 3.7]{LiebLoss2001}) which states that, for any measurable functions $f,g,h:\mathbb{R}^d\to [0,\infty)$ that vanish at infinity, we have
	\begin{equation}\label{eq-Ifgh}
		I(f,g,h)\le I(f^*,g^*,h^*),
	\end{equation}
	where $I(f,g,h):=\int_{\mathbb{R}^d}\int_{\mathbb{R}^d}f(x)g(x-y)h(y)\mathrm{d} x\mathrm{d} y$. 
	
	Now, let $f:=\mathds{1}_T$ and $h:=\mathds{1}_T$ be the indicator functions of $T$ and let $g(x):=\mathds{1}_{B_{2r_d}(0)}$ be the indicator function of the radius-$2r_d$ ball centred at the origin.
	Then $f^*$ and $h^*$ are the indicator functions of the ball $T^*$ while $g^*=g$. By~\eqref{eq-Ifgh} we have
	$$
	f(T)=I(f,g,h)\le I(f^*,g^*,h^*)=f(T^*),
	$$
	as required.
\end{proof}

With the help of~\cref{lem-ineq-f}, we can get the following strengthening of~\cite[Lemma~10]{JenssenJoosPerkins19} (which gives an upper bound on $\mathbb{E}_\mathbf{u}[\, \vol(B_{2r_d}(\mathbf{u})\cap T)\,]$ independent of $t$).

\begin{lemma}\label{lem-strength-exp-bound-intersec}
	Let $T\subseteq \mathbb{R}^d$ be a bounded measurable set of measure $t\in [2^{d/2},2^{d}]$. Let $\mathbf{u}$ be a random point chosen uniformly from $T$. Then,
	\begin{equation*}
		\mathbb{E}_\mathbf{u}[\, \vol(B_{2r_d}(\mathbf{u})\cap T)\,]\le 2\cdot 2^d\cdot (1-t^{-2/d})^{d/2}.
	\end{equation*}
\end{lemma}
\begin{proof}
	Note that $\mathbb{E}_u[ \,\vol(B_{2r_d}(\mathbf{u})\cap T)\,]=f(T)/\vol(T)$, and so, by Lemma~\ref{lem-ineq-f}, since $\vol(T)=\vol(T^*)$, it is enough to prove the lemma when $T=B_{\rho}(0)$ is a ball of radius $\rho:=t^{1/d}r_d$. This amounts to estimating a certain integral over $B_{\rho}(0)$ of a radially symmetric function. The following trick from~\cite{JenssenJoosPerkins19} simplifies calculations:
	\begin{equation*}
		\begin{split}
			\bE\big[\,\mathrm{vol}(B_{2r_d}(\mathbf{u})\cap T)\,\big]&=\frac{1}{t}\int_T\Bigg(\int_T \mathds{1}_{\{d(u,v)\leq 2r_d\}}\dd v\Bigg)\dd u\\
			&=\frac{2}{t}\int_T\int_T\mathds{1}_{\{d(u,v)\leq 2r_d\}}\cdot\mathds{1}_{\{\|v\|\leq\|u\|\}}\dd v\dd u\\
			&\leq2\max_{u\in B_\rho(0)}\int_T \mathds{1}_{\{d(u,v)\leq 2r_d\}}\cdot\mathds{1}_{\{\|v\|\leq\|u\|\}}\dd v\\
			&\leq 2\max_{u\in B_{\rho}(0)}\mathrm{vol}\big(B_{2r_d}(u)\cap B_{\|u\|}(0)\big).
		\end{split}
	\end{equation*}
	
	Let $u\in B_{\rho}(0)$ be a point maximising the volume of the intersection $\mathrm{vol}\big(B_{2r_d}(u)\cap B_{\|u\|}(0)\big)$, and let $x$ be such that $x\cdot r_d=\|u\|\le \rho$. If $x\le \sqrt{2}$, then the intersection has volume at most $\vol(B_{\sqrt{2}r_d}(0))\le 2^{d/2}$, which is at most the claimed bound. Suppose then that $x\ge \sqrt{2}$. Then, standard trigonometry shows that the intersection $B_{2r_d}(u)\cap B_{\|u\|}(0)$ is contained in a ball of radius $2\sqrt{1-x^{-2}}\cdot r_d$ centred at $(1-2/x^2)u$. Note that, by definition, $x\le t^{1/d}$. Hence, as $t\le 2^d$, we get that 
	$$
	\bE_{\mathbf u}[ \,\vol(B_{2r_d}(\mathbf{u})\cap T)\,] \le  
	\max\left\{ 2^{d/2},\, 2\cdot \max_{\sqrt{2}\le x\le t^{1/d}} (2\sqrt{1-x^{-2}})^d\right\}=
	2\cdot (2\sqrt{1-t^{-2/d}})^d,
	$$
	as required.
\end{proof}

We can now prove a lower bound for the expected number of centres of spheres in $T$ (which is equal to $\alpha_T(\lambda)\cdot\vol(T)$). This bound will be useful when the volume of $T$ is ``large''.

\begin{lemma}\label{lem-lower-bd-number-of-centres} For every $\beta>0$, there is $k_0$ such that, for every $k\ge k_0$ and every $\lambda, t,d>0$, if $T$ is a bounded measurable subset of $\mathbb{R}^d$ of measure $t$ and $k\in \mathbb{N}$ satisfies $k_0\le k\le \lambda t$, then
	\begin{equation}\label{eq-number-of-centres}
		\alpha_T(\lambda)\cdot\vol(T)\ge (1-\beta)p_kk,	
	\end{equation}
	where $p_i$ is the probability that for uniform independent points $x_1,\dots,x_i\in T$ every two are at distance at least $2r_d$.
\end{lemma}
\begin{proof}
	Let $X\sim \mu_{T,\lambda}$ be the random set produced by the hard sphere model of fugacity $\lambda$ on $T$ and let $x:=|X|$ be the number of spheres in this random packing of $T$. Thus, $\alpha_T(\lambda)\cdot\vol(T)=\bE[\,x\,]$ and we want to lower bound the expectation of $x$.
	
	If we take any random variable $Y$ and condition on $Y\le C$ then the expectation can only decrease; indeed, as $\bE[\,Y\,|\,Y\le C\,]\le C$, we have
	\begin{equation*}
		\begin{split}
			\bE[\,Y\,]&\ \ge\ \bE[\,Y\,|\,Y\le C\,]\cdot  \bP[\,Y\le C\,]+ C\cdot \bP[\,Y>C\,]\\
			&\ \geq\ \bE[\,Y\,|\,Y\le C\,]\cdot  \bP[\,Y\le C\,]+\bE[\,Y\,|\,Y\le C\,]\cdot\bP[\,Y>C\,]\ =\ \bE[\,Y\,|\,Y\le C\,].
		\end{split}
	\end{equation*}
	
	Thus it is enough to lower bound $\bE[\,x|x\le k\,]$. Let $\xi:=\lambda t$, $\gamma:=\beta/2$ and $m=\lceil(1-\gamma)k\rceil$. Since $1=p_0\ge p_1\ge\ldots\ge p_k$ and $\hat Z_T(i)=p_i\xi^i /i!$, we have
	\begin{equation*}
		\bE[\,x \,|\, x\le k\,] = \frac{\sum_{i=0}^k ip_i \xi^i/i!}{\sum_{i=0}^k p_i \xi^i/i!}\\
		\ge  \frac{p_k\sum_{i=m}^k i \xi^i/i!}{\sum_{i=0}^{k} \xi^i/i!}.
	\end{equation*}
	Let $N:=\sum_{i=m}^k i \,\xi^i/i!$. Thus $p_kN$ is the numerator of the last fraction. As $\xi^i/i!$ increases for $0\le i\le m-1$, we have
	$$
	\sum_{i=0}^{k} \xi^i/i!=\sum_{i=0}^{m-1}\xi^i/i!+\sum_{i=m}^{k} \xi^i/i!\le m\,\frac{\xi^{m-1}}{(m-1)!}+\frac1{m}  \sum_{i=m}^{k} i\,\xi^i/i! .
	$$ 
	Here, the first term is 
	$$
	m\,\frac{\xi^k}{k!}\frac{ \prod_{i=0}^{k-m} (1-\frac ik)}{(\xi/k)^{k-m+1}}\le m\,\frac{\xi^k}{k!}\, \me^{- \sum_{i=0}^{k-m} \frac ik}\le mN \me^{-\gamma^2 k/2},
	$$
	while the second term is exactly $\frac1m\, N$. Thus,
	$$
	\alpha_T(\lambda)\cdot\vol(T)\ge \bE[\, x \,|\, x\le k\, ]\ge \frac{p_k}{m\,\me^{-\gamma^2 k/2} + 1/m}\ge (1-\beta)p_kk,
	$$
	as required. 
\end{proof}

\subsection{Proof of~\cref{thm-improvement-sph-pack-dens}}
	Given $\eps>0$, choose sufficiently small constants $\beta,\delta$ so that $0<\delta \ll \beta \ll \eps$. Let $d\to\infty$ and take any measurable bounded $S\subseteq \mathbb{R}^d$ of positive measure. By~\cref{lem-interm-results-pack-dens-JJP}, $\alpha_S(\lambda)$ is an increasing function in~$\lambda$. So in order to prove Theorem~\ref{thm-improvement-sph-pack-dens}, it is enough to show that, for $\lambda=(1/\sqrt 2-\delta)^d$, we have
	\begin{equation}\label{eq-goal-improve}
		\alpha_S(\lambda)\ge (\log\sqrt2-\eps)\,d\cdot 2^{-d}.
	\end{equation}
	Let $X\sim \mu_{S,\lambda}$ be the centres of the sampled spheres.
	Take a point $v$ uniformly at random from $S$, independent of $X$. As in \cref{se:local}, let 
	$$
	\mathbf{T}=\mathbf{T}(X,v):=\{x\in B^\circ_{2r_d}(v)\cap S: d(x,y)\ge 2r_d, \, \forall y\in X\setminus B^\circ_{2r_d}(v)\}
	$$ 
	be the externally uncovered set around $v$ and let $t=t(X,v)$ be its measure.
	
	Let $k:=(\log \sqrt2-\eps/2)d$. 
	For $X\subseteq S$, let 
	$$
	L=L(X):=\{u\in S: t(X,u)\le k/\lambda\}.
	$$ 
	Now, note that from  Lemma~\ref{lem-interm-results-pack-dens-JJP}~$(\mathrm{iii})$ and  Lemma~\ref{lem-interm-results-pack-dens-JJP2}~$(\mathrm{i})$, we easily derive the inequality
	\begin{equation*}
		\alpha_S(\lambda)=\lambda\,\bE\Big[\,\frac{1}{Z_{\mathbf{T}}(\lambda)}\,\Big]\geq\lambda\, \bE_X \bE_v \big[\,\me^{-\lambda\cdot t(X,v)}\,\big].
    \end{equation*}	
    Then,
	\begin{align*}
	\alpha_S(\lambda)\cdot\vol(S)&\ge\lambda \,\bE_X \Big[\,\int_{v\in S} \me^{-\lambda\cdot t(X,v)}\dd v\,\Big]\\
	&\ge \lambda\, \bE_X \Big[\,\int_{v\in L} \me^{-\lambda\cdot t(X,v)}\dd v\,\Big]\\
	&\ge \me^{\eps d/3}\, 2^{-d}\, \bE_X [\,\vol(L)\,].
	\end{align*}
	Thus, we may assume that, for example, $\bE_X[\,\vol(L)\,]\le \vol(S)\,\me^{-\eps d/4}$, for otherwise~\eqref{eq-goal-improve} holds. Then, by Markov's inequality,
	\begin{equation}\label{eq-small-prob-t-large}
		\bP_X[\,\vol(L)\ge \vol(S)\,\me^{-\eps d/6}\,]\le \me^{-\eps d/12},
	\end{equation}
	that is, for typical outcome $X$, the measure $t=t(X,v)$ is ``relatively large'' except for a very small set of $v\in S$. 
	
	Take any $X$ with $\vol(L)\le \vol(S)\,\me^{-\eps d/6}$. For every $v\in S\setminus L$, $t=t(X,v)$ is at least $k/\lambda\ge (\sqrt 2+\delta/3)^{d}$ by the definition of $L$ and at most $2^d$ since $\mathbf{T}$ is a subset of $B^\circ_{2r_d}(v)$.  Again, by the definition of $L$, we have $k\le \lambda t$. 
	Thus, Lemma~\ref{lem-lower-bd-number-of-centres} applies and gives that, for every $v\in S\setminus L$, we have $\alpha_{\mathbf{T}}(\lambda)\cdot\vol(\mathbf{T})\ge (1-\beta) p_kk$, where  $p_i$ denotes the probability that, for uniform independent $x_1,\dots,x_i\in \mathbf{T}$, every two are at distance at least $2r_d$.

	\begin{claim}\label{cl:1-pk}
		$p_k\ge 1-\delta$.
	\end{claim}
	\begin{poc}
		Recall that $t\ge (\sqrt 2+\delta/3)^{d}$. Consider the function  $g(t):=(f(\tau))^{-d}$, where $\tau:=t^{1/d}$ and $f(\tau):=\tau/(2\sqrt{1-\tau^{-2}})$. Observe that $f(\sqrt{2})=1$ and $f$ is strictly increasing on $[\sqrt 2,2]$ since its derivative at $x\in (\sqrt 2,2]$ is
		$$
		f'(x)=\frac{1}{2
			\sqrt{1-\frac{1}{x^2}}}-\frac{1}{
			2
			\left(1-\frac{1}{x^2}\right)^{3/2
			} x^2}=\frac{x^2-2}{2
			\sqrt{1-\frac{1}{x^2}}
			\left(x^2-1\right)}>0. $$
		Thus $f(t^{1/d})\ge f(\sqrt2+\delta/3)>1$, which means that $g(t)$ is exponentially small in $d\to\infty$. Note that, by Lemma~\ref{lem-strength-exp-bound-intersec}, $2g(t)$ upper bounds the expected fraction of measure of $\mathbf{T}$ that a ball of radius $2r_d$ centred at a uniformly chosen random point $x\in \mathbf{T}$ covers. 
		
		Let $x_1,\dots,x_k\in \mathbf{T}$ be independent random points chosen according to the uniform distribution on $\mathbf{T}$. Imagine that we sample $x_1$, then $x_2$, and so on.
		Call $x_i$ \emph{bad} if $\vol(B_{2r_d}(x_i)\cap \mathbf{T})\ge t/d^3$ or $x_i$ is within distance $2r_d$ from any of $x_1,\dots,x_{i-1}$; otherwise call $x_i$ \emph{good}. Clearly, if all vertices are good then every two are at distance larger than $2r_d$, so it is enough to show that the probability of at least one $x_i$ being bad is $o(1)$ as $d\to\infty$. 
		
		Thus it is enough to show that, for each $i$, the probability that $x_i$ is the first bad vertex is $o(1/k)$. Indeed, this probability (i.e. that $x_i$ is the first bad vertex) is at most the probability that
		$\vol(B_{2r_d}(x_i)\cap \mathbf{T})\ge t/d^3$, which is exponentially small in $d$ by Markov's inequality, plus
		the probability that it belongs to the forbidden region of the good vertices $x_1,\dots,x_{i-1}$, which is at most $(i-1)/d^3\le k/d^3=o(1/k)$, proving the claim.	
	\end{poc}
	
	Therefore, by Lemma~\ref{lem-lower-bd-number-of-centres}, we have, for every $v\in S\setminus L$, that 
	\begin{equation}\label{eq-num-centres-proof}
		\alpha_{\mathbf{T}}(\lambda)\cdot\vol(\mathbf{T})\geq(1-\beta) p_kk\ge (1-2\beta)k.
	\end{equation}
	Then,~\cref{lem-interm-results-pack-dens-JJP2}~(ii) gives that
	\begin{equation}
		2^d \alpha_S(\lambda)\ge \bE_{X,v}\big[\,\alpha_\mathbf{T}(\lambda)\cdot\vol(\mathbf{T})\,\big]=\frac{1}{\vol(S)}\,\bE_X\left[\,\int _{v\in S} \alpha_{\mathbf{T}}(\lambda)\cdot\vol(\mathbf{T})\dd v\,\right].	
	\end{equation}
	So, we have, by~\eqref{eq-num-centres-proof}, that
	\begin{equation*}
		\begin{split}
			2^d \alpha_S(\lambda)&\ge\frac{1}{\vol(S)}\, \bE_X\left[\,\int _{v\in S\setminus L} \alpha_{\mathbf{T}}(\lambda)\cdot\vol(\mathbf{T})\dd v\,\right]\\
			&\ge (1-\me^{-\eps d/12})(1-\me^{-\eps d/6}) (1-2\beta)k\\
			&\ge (\log\sqrt2-\eps)d,
		\end{split}
	\end{equation*}
	where the second inequality is obtained by taking only $X$ with $\vol(L)\le \vol(S)\,\me^{-\eps d/6}$ and using~\eqref{eq-small-prob-t-large}. This proves~\eqref{eq-goal-improve}, thus finishing the proof of~\cref{thm-improvement-sph-pack-dens}.\medskip


\section{Kissing numbers and spherical codes in high dimensions}\label{sec-improve-kissing}

We devote this section to showing how the same method can be used to improve the lower bound on kissing numbers and spherical codes in $\bR^d$.  For this, we first introduce the so called \emph{hard cap model}.

\subsection{The hard cap model}\label{sec-hard-cap-model}

The \emph{hard cap model} is an analogue of the hard sphere model but, instead of considering sphere packings, we are interested in packing the surface of a unit ball with spherical caps. 

Let $P_k(d,\theta)$ be the set of all spherical codes of size $k$ and angle $\theta$ in dimension $d$ (or, equivalently, the set of centres of non-overlapping spherical caps of angular radius $\theta/2$), that is,
$$
P_k(d,\theta):=\{\{x_1,\dots,x_k\}\subset \Sd:\, \forall i\neq j\ \ \langle x_i,x_j\rangle\leq\cos\theta\}.
$$
Similar to the hard sphere model (see Section~\ref{sec-hard-sphere-model}), the hard cap model is a probability distribution over configurations of non-overlapping, identical spherical caps in $\Sd$, and there are two versions depending on whether the number of spherical caps is given or it is random. 

In the \emph{canonical hard cap model}, we are given a non-negative integer $k$ that is the size of a spherical code in $\Sd$ and a spherical code from~$P_k(d,\theta)$ is sampled uniformly at random (if $P_k(d,\theta)$ has positive measure). The partition function is given by
$$\hat{Z}_d^{\theta}(k):=\frac{1}{k!}\int_{\Sd^k}\mathds{1}_{\mathcal{D}_\theta(x_1,\dots,x_k)}\dd s(x_1)\cdots \dd s(x_k),\mbox{ for $k\ge 1$,} \quad\mbox{and}\quad \hat{Z}_d^{\theta}(0):=1,$$
where $\mathcal{D}_\theta(x_1,\dots,x_k)$ is the event that $\langle x_i,x_j\rangle\leq\cos\theta$, for every $1\leq i<j\le k$, and the integrals over $\Sd$ are with respect to the normalised surface measure $s(\cdot)$. Note that if $X_k=(x_1,\dots,x_k)\in \Sd^k$ is chosen uniformly at random, then
$$\bP[\,X_k\in P_k(d,\theta)\,]=k!\cdot \hat{Z}_d^\theta(k).$$

In the \emph{grand canonical hard cap model at fugacity $\lambda$}, we sample $X$ according to a Poisson point process of intensity $\lambda$ on $\Sd$ conditioned on the event that $\langle x,y\rangle\leq\cos\theta$, for every distinct $x,y\in X$. 
As in Lemma~\ref{lm:equiv}, this distribution can be equivalently described  by the \emph{partition function}
$$
Z_d^\theta(\lambda):=\sum_{k=0}^\infty\lambda^k\hat{Z}_d^\theta(k),
$$
 where we first pick an integer $k$ with probability $\hat{Z}_d^\theta(k)/Z_d^\theta(\lambda)$ and then take $k$ independent uniform points in $\Sd$ conditioned on the minimum angular distance being at least~$\theta$.
 
We shall write $\mu_{d,\lambda}^{\theta}$ for the probability measure of the hard cap model with angle $\theta$ at fugacity~$\lambda$ on $\Sd$, and abbreviate $\bP_{X\sim\mu_{d,\lambda}^{\theta}}$ to $\bP_{X}$ or simply $\bP$ when there is no confusion.



We can similarly define the expected size of a random spherical code to lower bound the kissing number $K(d)$ and the maximum size of spherical codes. More precisely, define
	$$\alpha_d^\theta(\lambda):=\bE_{X\sim \mu_{d,\lambda}^{\theta}}[\,|X|\,],$$
which is a lower bound on $A(d,\theta)$.

We can also define the hard cap model on any measurable set $A\subseteq \Sd$ with partition function $Z_A^\theta(\lambda):=\sum_{k=0}^\infty\lambda^k\hat{Z}_A^\theta(k),$ where $\hat{Z}_A^\theta(k):=\frac{1}{k!}\int_{A^k}\mathds{1}_{\mathcal{D}_\theta(x_1,\dots,x_k)}\dd s(x_1)\ldots \dd s(x_k)$. 
We  write $\alpha_A^\theta(\lambda)$ for the expected size of such a random spherical code sampled on $A$.


Let $q(\theta)$ be the angular radius of the smallest spherical cap that contains the intersection of two spherical caps of angular radius $\theta$ whose centres are at angle $\theta$. A special case of Lemma~\ref{lem-angle-intersec-2caps} below gives that
\begin{equation}\label{defn:q}
q(\theta)=\arcsin \Big(\frac{(1-\cos\theta)\sqrt{1+2\cos\theta}}{\sin\theta}\Big).
\end{equation}

The following lower bound on the average size of a random spherical code was obtained by Jenssen, Joos and Perkins~\cite{JenssenJoosPerkins18}.

\begin{theorem}[Theorem 4, \cite{JenssenJoosPerkins18}]\label{thm:sph-code-JJP}
	Let $\theta\in(0,\pi/2)$ be fixed and $q(\theta)$ be as in~\cref{defn:q}. Then, for $\lambda\geq \frac{1}{d\cdot s_d(q(\theta))}$,
	$$\alpha_d^\theta(\lambda)\geq (1+o(1))\log\frac{\sin\theta}{\sin q(\theta)}\cdot d\cdot s_d(\theta)^{-1}.$$
\end{theorem}

Our improvement 
over~\cref{thm:sph-code-JJP} (on a smaller range of $\lambda$) is given by the following result, which implies~\cref{thm:sph-code}.

\begin{theorem}\label{thm-improve-kissing}
	For every $\eps>0$ there are $\delta>0$ and $d_0$ such that if $d\ge d_0$ then the expected density of a hard cap model with fugacity $\lambda\ge \big(\sqrt{2}\sin\frac{\theta}{2}+\delta\big)^{-d}$ in $ \Sd$ satisfies
	\begin{equation}\label{eq-aim-improve-kissing}
		\alpha_{d}^{\theta}(\lambda)\ge \Big(\log\frac{\sin\theta}{\sqrt{2}\sin\frac{\theta}{2}}-\eps\Big)d\cdot s_d(\theta)^{-1}.
	\end{equation}
	In particular, 
	\cref{thm:sph-code} holds.
\end{theorem}

\subsection{Proof of Theorem~\ref{thm-improve-kissing}}\label{sec-pf-improve-kissing}
Similarly as we did in Section~\ref{sec-improve-hard-sph}, consider a two-part experiment as follows: sample a random spherical code $X\sim\mu_{d,\lambda}^{\theta}$ on $\Sd$ and, independently, choose a point $v$ uniformly at random from $\Sd$. We analogously define the externally uncovered neighbourhood around $v$ as
\begin{equation}\label{defn:T-cap}
\mathbf{T}=\mathbf{T}(X,v):=\{x\in C^\circ_\theta(v):\,\forall y\in X\setminus C^\circ_\theta(v)
 \ \ \langle x,y\rangle\leq\cos\theta\}.
\end{equation}

In the proof of Theorem~\ref{thm-improve-kissing}, we make use of several properties of $\alpha_{d}^{\theta}(\lambda)$ that were proved in~\cite[Lemma~5]{JenssenJoosPerkins18}. We put them here for the reader's convenience; we omit their proofs, as they are analogous to those of Lemmas~\ref{lem-interm-results-pack-dens-JJP} and \ref{lem-interm-results-pack-dens-JJP2}.

\begin{lemma}\label{lem-properties-density-hard-cap}
	Let $\lambda>0,\theta\in(0,\pi/2), X\sim \mu_{d,\lambda}^{\theta}$ and let $v\in\Sd$ be a  point chosen uniformly at random from $\Sd$, independent of $X$. Let $\mathbf{T}$ be as in~\cref{defn:T-cap}. Then, the followings hold:
	\begin{enumerate}
	    \item[$(\mathrm{i})$] $\alpha_d^\theta(\lambda)=\lambda\cdot \big(\log Z_d^\theta(\lambda)\big)'$;
	    \item[$(\mathrm{ii})$] $\alpha_d^\theta(\lambda)$ is strictly increasing in $\lambda$;
	    \item[$(\mathrm{iii})$] $\alpha_d^\theta(\lambda)=\lambda\cdot \bE_{X,v}\Big[\frac{1}{Z_{\mathbf{T}}(\lambda)}\Big];$
		\item[$(\mathrm{iv})$]  $\alpha_d^\theta(\lambda)\geq\lambda\cdot \bE_{X,v} \Big[\me^{-\lambda\cdot s(\mathbf{T})}\Big];$
		\item[$(\mathrm{v})$] $\alpha_d^\theta(\lambda)=\frac{1}{s_d(\theta)}\cdot\bE_{X,v}[\alpha_{\mathbf{T}}^\theta(\lambda)] $.
	\end{enumerate}
\end{lemma}

The next ingredient we need is the analogous result to Lemma~\ref{lem-strength-exp-bound-intersec}. To state it, we need some preliminaries. For $\theta\in(0,\pi/2)$, define $\theta'\in (0,\pi/2)$ by
\begin{equation}\label{eq:optimal-radius}
	\sin\theta'=\sqrt{2}\,\sin\frac{\theta}{2}.
\end{equation}
 Note that $\sin \theta=2\sin(\theta/2)\sqrt{1-\sin^2(\theta/2)}\ge \sqrt2\,\sin \theta' \, \sqrt{1-\sin^2(\pi/4)}\ge \sin\theta'$ and thus $\theta'\le \theta$.


\begin{lemma}\label{lem-angle-intersec-2caps}
Let $\theta\in(0,\pi/2)$ and $\tau\in[\theta',\theta]$, where $\theta'$ is as in~\cref{eq:optimal-radius}. Let $x,u\in \Sd$ be two points with an angle of~$\tau$ apart. Then, the intersection of the two caps $C_{\tau}(x)\cap C_{\theta}(u)$ is contained in a spherical cap of angular radius
\begin{equation}\label{eq:cap-inter-angl-radius}
	\sigma(\tau,\theta):=\arcsin\left(\frac{\sqrt{1+2\cos^2\tau\cos\theta-2\cos^2\tau-\cos^2\theta}}{\sin\tau}\right).
\end{equation}
Moreover, 
$\sigma(\tau,\theta)$ is increasing in $\tau$.
\end{lemma}
\begin{proof}
The first claim
 was established in the proof of~\cite[Lemma 6]{JenssenJoosPerkins18}, namely see~\cite[Equation~(10)]{JenssenJoosPerkins18}. Observe that the expression under the square root in~\cref{eq:cap-inter-angl-radius} is
$$
1+2\cos^2\tau\cos\theta-2\cos^2\tau-\cos^2\theta=(1-\cos\theta)(1+\cos\theta-2\cos^2\tau);
$$
 this is non-negative, which follows from $\tau\in[\theta',\theta]$ and the choice of $\theta'$.

For the second part, 
write $\sigma(\tau,\theta)=\arcsin(f_\theta(\tau))$, where 
$$
f_\theta(\tau):=\frac{\sqrt{1+2\cos^2\tau\cos\theta-2\cos^2\tau-\cos^2\theta}}{\sin\tau}.
$$
We need to show that the partial derivative of $\sigma(\tau,\theta)$ with respect to $\tau$ is positive, for which it suffices to prove that $f'_\theta(\tau)> 0$. One can check that
\begin{equation*}
        f'_\theta(\tau)
        =\frac{\cos\tau\cdot(1-\cos\theta)^2}{\sin^2\tau\sqrt{(1-\cos\theta)(1+\cos\theta-2\cos^2\tau)}}.
\end{equation*}
The numerator above is clearly positive as $\tau,\theta\in(0,\pi/2)$.
\end{proof}

Note that, by~\cref{eq:cap-inter-angl-radius}, $\sin^2\tau\cdot (\sin^2\tau-\sin^2\sigma(\tau,\theta))$ simplifies to $(\cos^2\tau-\cos\theta)^2= (\sin^2\tau - 2\sin^2\frac{\theta}{2})^2$ and so
\begin{equation}\label{eq:optimal-radius2}
	\sin\sigma(\tau,\theta)\le \sin\tau, \quad \text{ with equality if and only if $\tau=\theta'$}.
\end{equation}
 
We also need an analogous result to~\cref{lem-ineq-f}. It can be derived using spherical rearrangements and the analogue of Riesz's rearrangement inequality for spheres proved by Baernstein~\cite[Theorem 2]{BaernsteinTaylor76dmj}.

\begin{lemma}\label{lem:sph-rearrange} Let $T$ be a bounded measurable set in $\Sd$ and let $f(T):=\int_{T}s(C_{\theta}(u)\cap T)\,\mathrm{d}u$. Let $T^*$ be a spherical cap with the same measure as $T$. Then, 
	\begin{equation*}
		f(T)\le f(T^*).
	\end{equation*}	
\end{lemma}

The following result  is a strengthening of Lemma~6 from~\cite{JenssenJoosPerkins18} (whose authors use the upper bound $2\cdot s_d(q(\theta))$, independent of the measure of $T$).

\begin{lemma}\label{lm:cap-intersection} 
	Let $\theta\in(0,\pi/2)$, $\theta'$ and $\sigma(\cdot,\cdot)$ be as in~\cref{eq:cap-inter-angl-radius,eq:optimal-radius} respectively. Let $T\subseteq \Sd$ be a bounded measurable set such that $s(T) = s_d(\alpha)$ with $\alpha \in [\theta',\theta]$
and let $\mathbf{u}$ be a uniform point of $T$. Then
	\begin{equation}
		\bE_\mathbf{u}[\, s(C_{\theta}(\mathbf{u})\cap T)\, ]\le 2\cdot s_d(\sigma(\alpha,\theta)).
	\end{equation}
\end{lemma}
\begin{proof}
      By~\cref{lem:sph-rearrange}, we may assume that $T=C_{\alpha}(x)$ is a spherical cap with angular radius $\alpha\in[\theta',\theta]$ centred at some point $x$. Now we compute
	\begin{equation*}
		\begin{split}
			\bE_\mathbf{u}[ \,s(C_{\theta}(\mathbf{u})\cap T)\, ] &= \frac{1}{s(T)} \int_{T^2} \mathds{1}_{\langle u,v\rangle\ge\cos\theta}\dd s(u)\dd s(v)\\
			&= \frac{2}{s(T)} \int_{T^2} \mathds{1}_{\langle u,v\rangle\ge\cos\theta}\cdot \mathds{1}_{\langle x,v\rangle\ge \langle x,u\rangle}\dd s(u)\dd s(v)\\
			&\le 2 \max_{u\in C_{\alpha}(x)} \int_T\mathds{1}_{\langle u,v\rangle\ge\cos\theta}\cdot \mathds{1}_{\langle x,v\rangle\ge \langle x,u\rangle}\dd s(v)\\
			&\le  2 \max_{u\in C_{\alpha}(x)} s(C_{\theta}(u)\cap C_{\gamma}(x)),
		\end{split}
	\end{equation*}
	where $\gamma:=\arccos\langle x,u\rangle\in[0,\alpha]$. 
	
	If $\gamma<\theta'$, we can bound
	 $$s(C_{\theta}(u)\cap C_{\gamma}(x))\le s(C_{\theta'}(x))=s_{d}(\theta').$$
	 Note that this is at most $s_d(\sigma(\alpha,\theta))$ because $\theta'=\sigma(\theta',\theta)$ due to~\cref{eq:optimal-radius2} while 
	 $\sigma(\gamma,\theta)$ is increasing when $\gamma\in [\theta',\alpha]$.
	 
	If $\theta'\le \gamma\le\alpha$, we can bound
	$$s(C_{\theta}(u)\cap C_{\gamma}(x))\le s_d(\sigma(\gamma,\theta))\le s_d(\sigma(\alpha,\theta)),$$ 
	where the first inequality follows from the definition of $\sigma(\cdot,\cdot)$ and the second one is because $\sigma(\gamma,\theta)$ is increasing when $\gamma\in [\theta',\alpha]$ due to~\cref{lem-angle-intersec-2caps}.	
\end{proof}

With identical proof as Lemma~\ref{lem-lower-bd-number-of-centres}, we get the following version for spherical codes.
\begin{lemma}\label{lm:pois-mean-sph} Let $\theta\in (0,\pi/2)$. For every $\beta>0$ there is $k_0$ such that for every $k\ge k_0$ and every $\lambda, t,d>0$, if $T$ is a measurable subset of $\Sd$ of measure $t$ and $k\in \mathbb{N}$ satisfies $k_0\le k\le \lambda t$, then
	\begin{equation}\label{eq:k}
		\alpha_T^{\theta}(\lambda)\ge (1-\beta)p_kk,
	\end{equation}
	where $p_i$ is the probability that for uniform independent $x_1,\dots,x_i\in T$ the angle between every two is at least $\theta$.
\end{lemma}

We are now ready to prove~\cref{thm-improve-kissing}.
\begin{proof}[Proof of~\cref{thm-improve-kissing}]
	Given $\eps>0$, choose sufficiently small constants $\beta,\delta$ so that $0<\delta\ll \beta \ll \eps$. Let $d\to\infty$. By~\cref{lem-properties-density-hard-cap},  $\alpha_{d}^{\theta}(\lambda)$ is strictly increasing function in $\lambda$, so it is enough to prove Theorem~\ref{thm-improve-kissing} for $\lambda=\big(\sqrt{2}\sin\frac{\theta}{2}+\delta\big)^{-d}$.
	
	Let $X\sim \mu_{d,\lambda}^{\theta}$ be the set of centres of the caps sampled on $\Sd$ according to the hard cap model with fugacity $\lambda$. Take a point $v\in \Sd$ uniformly at random from $\Sd$, independently of~$X$. As before, let 
	$$
	\mathbf{T}=\mathbf{T}(X,v):=\big\{y\in C^\circ_{\theta}(v): \forall x\in X\setminus C^\circ_{\theta}(v)\ \ \langle x,y\rangle\le \cos\theta\big\}
	$$ 
	be the externally uncovered set around $v$ and let $t:=t(X,v)$ be its measure.
	
	Let $k:=\Big(\log\frac{\sin\theta}{\sqrt{2}\sin\frac{\theta}{2}}-\frac{\eps}{2}\Big)d$. 
	For $X\subseteq \Sd$, let 
	$$
	L=L(X):=\{u\in \Sd: t(X,u)\le k/\lambda\}.
	$$ 
	By~\cref{lem-properties-density-hard-cap} (iv), we have that
	\begin{align*}
		\alpha_{d}^{\theta}(\lambda)&\ge \lambda\, \bE_X \bE_v \big[\,\me^{-\lambda\cdot t(X,v)}\,\big]\ge\lambda\, \bE_X \Big[\,\int_{v\in  \Sd} \me^{-\lambda\cdot t(X,v)}\dd s(v)\Big]\\
		&\ge \lambda\, \bE_X \Big[\,\int_{v\in L} \me^{-\lambda\cdot t(X,v)}\dd s(v)\,\Big]\ge  \lambda\, \bE_X \Big[\,\int_{v\in L} \me^{-k}\dd s(v)\,\Big]\\
		&\ge \me^{\eps d/3}\, (\sin\theta)^{-d}\cdot \bE_X [\,s(L)\,] \ge \me^{\eps d/4}\, s_d(\theta)^{-1}\cdot \bE_X [\,s(L)\,] .
	\end{align*}
	Thus we may assume that, for example, $\bE_X[\,s(L)\,]\le \me^{-\eps d/5}$, for otherwise~\eqref{eq-aim-improve-kissing} holds. By Markov's inequality,
	\begin{equation}\label{eq:VolL}
		\bP_X[\,s(L)\ge \me^{-\eps d/6}\,]\le \me^{-\eps d/30},
	\end{equation}
	that is, for typical outcome $X$, $t$ is ``relatively large'' except for a very small set of $v\in \Sd$.
	
	Take any $X$ with $s(L)\le \me^{-\eps d/6}$. For every $v\in \Sd\setminus L$, $t=t(X,v)$ is at least 
	$$\frac{k}{\lambda}\ge \Big(\sqrt{2}\sin\frac{\theta}{2}+\frac{\delta}{2}\Big)^{d}$$ 
	by the definition of $L$ and at most $s_d(\theta)$ since $\mathbf{T}$ is a subset of $C_{\theta}(v)$.  Again, by the definition of $L$, we have $k\le \lambda t$. 
	Thus Lemma~\ref{lm:pois-mean-sph} applies and gives that for every $v\in \Sd\setminus L$ we have $\alpha_{T(v)}^{\theta}(\lambda)\ge (1-\beta) p_kk$, where  $p_k$ denotes the probability that, for uniform independent $x_1,\dots,x_k\in \mathbf{T}$, the angle between every two is at least $\theta$.

	\begin{claim}\label{claim-pk-kissing}
		$p_k\ge 1-\delta$.
	\end{claim}
	\begin{poc}
		Let $\alpha>0$ be such that $t=s_d(\alpha)$. Then, as $t\ge \big(\sqrt{2}\sin\frac{\theta}{2}+\frac{\delta}{2}\big)^{d}$, we see that $\alpha>\theta'$, where $\theta'$ is as defined in~\eqref{eq:optimal-radius}. Consequently, by~\eqref{eq:optimal-radius2}, we have $\sin\sigma(\alpha,\theta)<\sin\alpha$, where $\sigma(\cdot,\cdot)$ is as in~\eqref{eq:cap-inter-angl-radius}. Thus, by Lemma~\ref{lm:cap-intersection} and \eqref{eq:optimal-radius2}, we get that 
		$$	\bE_u[ \,s(C_{\theta}(u)\cap \mathbf{T}) \,]\le 2\cdot s_d(\sigma(\alpha,\theta))<\me^{-\Omega(d)}\cdot s_d(\alpha)=\me^{-\Omega(d)}\cdot t.$$
		That is, the expected fraction of measure of $\mathbf{T}$ that cap of angular radius $\theta$ at a uniform $u\in \mathbf{T}$ covers is exponentially small in $d\to\infty$. Now we finish the proof the same way in Claim~\ref{cl:1-pk}: basically, since each point `forbids' $o(1/k^2)$-fraction of volume of $\mathbf T$, the probability that some two points conflict with each other is $o(1)$.
	\end{poc}
	Therefore, by Lemma~\ref{lm:pois-mean-sph} and Claim~\ref{claim-pk-kissing}, we have, for every $v\in \Sd\setminus L$, that
	\begin{equation*}
		\alpha_{\mathbf{T}(v)}^{\theta}(\lambda)\ge (1-\beta) p_kk\ge (1-2\beta)k.
	\end{equation*}
	This together with~\cref{lem-properties-density-hard-cap} (v) gives that
	\begin{align*}
			s_d(\theta)\cdot  \alpha_d^{\theta}(\lambda)&= \bE_{X,v}[\,\alpha_\mathbf{T}^{\theta}(\lambda)\,]=\bE_X\left[\,\int _{v\in \Sd} \alpha_{\mathbf{T}(v)}^{\theta}(\lambda)\dd s(v)\,\right]\ge\bE_X\left[\,\int _{v\in \Sd\setminus L} (1-2\beta)k\dd s(v)\,\right]\\
			&\ge (1-\me^{-\eps d/30})(1-\me^{-\eps d/6}) (1-2\beta)k\ge \Big(\log\frac{\sin\theta}{\sqrt{2}\sin\frac{\theta}{2}}-\eps\Big)d,	
		\end{align*}
	where the second inequality is obtained by taking only $X$ with $s(L)\le \me^{-\eps d/6}$ and using~\eqref{eq:VolL}. This proves Theorem~\ref{thm-improve-kissing}.
\end{proof}

\end{document}